\documentclass[12pt,a4paper]{article}
\usepackage[utf8x]{inputenc}
\usepackage{amsmath}
\usepackage{amsfonts}
\usepackage{amssymb}
\usepackage{amsthm}
\usepackage{mathrsfs}
\usepackage{fancyhdr}
\usepackage{graphicx}
\usepackage[usenames,x11names]{xcolor}
\usepackage{arabtex}
\usepackage{framed}
\usepackage[top=2cm,bottom=2cm,margin=2cm]{geometry}
\usepackage{cancel}
\usepackage{array}   

\usepackage[dvipdfm,pdfstartview=FitH,colorlinks=true,linkcolor=Red4,urlcolor=Red4,%
filecolor=green,citecolor=red]{hyperref}

\title{Nontrivial effective lower bounds for the least common multiple of some quadratic sequences}
\author{\textsc{Sid Ali BOUSLA} and \textsc{Bakir FARHI} \\
Laboratoire de Mathématiques appliquées \\
Faculté des Sciences Exactes \\
Université de Bejaia, 06000 Bejaia, Algeria \\[1mm]
\href{mailto:bouslasidali@gmail.com}{bouslasidali@gmail.com} (S.A. Bousla),  \href{mailto:bakir.farhi@gmail.com}{bakir.farhi@gmail.com} (B. Farhi)
}
\date{}

\def\N{{\mathbb N}}
\def\Z{{\mathbb Z}}

\def\lcm{\mathrm{lcm}}

\def\EMdash{\leavevmode\hbox to 10.6mm{\vrule height .63ex depth -.59ex
    width 10mm\hfill}}

\theoremstyle{plain}
\numberwithin{equation}{section}
\newtheorem{thm}{Theorem}[section]
\newtheorem{theorem}[thm]{Theorem}
\newtheorem{lemma}[thm]{Lemma}

\newtheorem{prop}[thm]{Proposition}
\newtheorem{coll}[thm]{Corollary}

\begin{document}

\maketitle
\begin{abstract}
This paper is devoted to studying the numbers $L_{c,m,n} := \mathrm{lcm}\{m^2+c,(m+1)^2+c,$ \linebreak $\dots,n^2+c\}$, where $c,m,n$ are positive integers such that $m \leq n$. Precisely, we prove that $L_{c,m,n}$ is a multiple of the rational number
\[\frac{\displaystyle\prod_{k=m}^{n}\left(k^2+c\right)}{c \cdot (n-m)!\displaystyle\prod_{k=1}^{n-m}\left(k^2+4c\right)} ,\]
and we derive (as consequences) some nontrivial lower bounds for $L_{c,m,n}$. We prove for example that if $n- \frac{1}{2} n^{2/3} \leq m \leq n$, then we have $L_{c,m,n} \geq \lambda(c) \cdot n e^{3 (n - m)}$, where $\lambda(c) := \frac{e^{- \frac{2 \pi^2}{3} c - \frac{5}{12}}}{(2 \pi)^{3/2} c}$. Further, it must be noted that our approach (focusing on commutative algebra) is new and different from those using previously by Farhi, Oon and Hong. 
\end{abstract}

\section{Introduction and Notation}
Throughout this paper, we let $\mathbb{N^*}$ denote the set $\mathbb{N}\setminus\lbrace 0\rbrace$ of positive integers. For $t\in\mathbb{R}$, we let $\lfloor t\rfloor$ and $\lceil t\rceil$ respectively denote the floor and the ceiling function. We say that an integer $a$ is a multiple of a non-zero rational number $r$ (or equivalently, $r$ is a divisor of $a$) if the ratio $a / r$ is an integer. If $m,n,c$ are positive integers such that $m\leq n$, we set $L_{c,m,n}:=\lcm\left\lbrace m^2+c,(m+1)^2+c,\dots,n^2+c \right\rbrace$. For a given polynomial $P\in\mathbb{C}[X]$, we denote by $\overline{P}$ its polynomial conjugate in $\mathbb{C}[X]$, that is the polynomial we get by replacing each coefficient of $P$ with its complex conjugate. It is well-known that the conjugation of polynomials in $\mathbb{C}[X]$ is compatible with addition and multiplication, in the sens that for every $P,Q\in\mathbb{C}[X]$, we have $\overline{P+Q}=\overline{P}+\overline{Q}$ and $\overline{P\cdot Q}=\overline{P}\cdot\overline{Q}$. Further, we let $I$, $E_h$ $(h\in\mathbb{R})$ and $\Delta$ denote the linear operators on $\mathbb{C}[X]$ which respectively represent the identity, the shift operator with step $h$ ($E_hP\left(X\right)=P\left(X+h\right)$, $\forall P\in\mathbb{C}[X]$) and the forward difference ($\Delta P\left(X\right)=P\left(X+1\right)-P\left(X\right)$, $\forall P\in\mathbb{C}[X]$). For $n\in\mathbb{N}$, the expression of $\Delta^n$ in terms of the $E_h$'s is easily obtained from the binomial formula, as follows:
\begin{equation}\label{dece}
\Delta^n=\left(E_1-I\right)^n=\sum_{m=0}^{n}(-1)^{n-m}\binom{n}{m}E_{1}^m=\sum_{m=0}^{n}(-1)^{n-m}\binom{n}{m}E_{m}.
\end{equation} 
For falling factorial powers, we use Knuth's notation:
\[X^{\underline{n}}:=X\left(X-1\right)\left(X-2\right)\cdots\left(X-n+1\right)~~~~(\forall n\in\mathbb{N}).\] 
The study of the least common multiple of the first $n$ consecutive positive integers $(n\in\mathbb{N^*})$ began with Chebyshev's work \cite{cheb} in his attempts to prove the prime number theorem. The latter showed that the prime number theorem is equivalent to stating that $\log \lcm\left(1,2,\dots,n\right)\sim_{+\infty} n$. More recently, many authors are interested in the effective estimates of the least common multiple of consecutive terms of some integer sequences. In 1972, Hanson \cite{han} showed (by leaning on the development of the number 1 in Sylvester series) that $\lcm\left(1,2,\dots,n\right)\leq 3^n$ $(\forall n\in\mathbb{N^*})$. In 1982, investigating the integral $\int_{0}^{1}x^n(1-x)^n\mathrm{d}x$, Nair \cite{nair} gave a simple proof that $\lcm\left(1,2,\dots,n\right)\geq 2^n$ $(\forall n\geq 7)$. In the continuation, the second author \cite{far} obtained nontrivial lower bounds for the least common multiple of consecutive terms in an arithmetic progression. In particular, he proved that for any $u_0,r,n\in \mathbb{N^*}$ such that $\gcd\left(u_0,r\right)=1$, we have:
\begin{equation}\label{01}
\lcm\left(u_0,u_0+r,\dots,u_0+nr\right)\geq u_0\left(r+1\right)^{n-1},
\end{equation}
and conjectured that the exponent $(n-1)$ appearing in the right-hand side of \eqref{01} can be replaced by $n$, which is the optimal exponent that can be obtained. That conjecture was confirmed by Hong and Feng \cite{hong1}. Furthermore, several authors obtained improvements of \eqref{01} for $n$ sufficiently large in terms of $u_{0}$ and $r$ (see e.g., \cite{hong1},\cite{hongkom} and \cite{kane}). The second author \cite{far} also obtained nontrivial lower bounds for some quadratic sequences. In particular, he proved that for any positive integer $n$, we have:
\begin{equation}\label{02+}
\lcm\left(1^2+1,2^2+1,\dots,n^2+1\right)\geq 0.32(1.442)^n.
\end{equation}
In 2013, Oon \cite{oon} managed to improve \eqref{02+} by proving that for any positive integers $c$ and $n$, we have:     
\begin{equation}\label{02}
\lcm\left(1^2+c,2^2+c,\dots,n^2+c\right)\geq 2^n.
\end{equation}
Actually, we have the following result a little stronger:
\begin{theorem}[Oon \cite{oon}]\label{03}
Let $c,n,m$ be positive integers such that $m\leq \left\lceil \frac{n}{2}\right\rceil$. Then, we have:
\[L_{c,m,n}\geq 2^n.\] 
\end{theorem}
In the next, Hong et al. \cite{hong2} successed to generalize Theorem \ref{03} for polynomial sequences $\left(f(n)\right)_{n\geq 1}$, with $f\in\mathbb{Z}[X]$ and the coefficients of $f$ are all nonnegative. In another direction, various asymptotic estimates have been obtained by several authors. For example, Bateman \cite{bat} proved that for any $h,k\in\mathbb{Z}$ with $k>0$, $h+k>0$ and $\gcd(h,k)=1$, we have: 
\begin{equation}\label{04}
\log\lcm\lbrace h+k,h+2k,\dots,h+nk\rbrace ~\sim_{+ \infty}~ \left(\frac{k}{\varphi(k)}  \sum_{\begin{subarray}{c} 1\leq m\leq k \\ \gcd(m , k)=1  \end{subarray}}\frac{1}{m}\right) n ,
\end{equation} 
where $\varphi$ denotes the Euler totient function. Another asymptotic estimate a little harder to prove is due to Cilleruelo \cite{cil} and states that for every irreducible quadratic polynomial $f\in\mathbb{Z}[X]$, we have: 
\begin{equation}\label{05}
\log \lcm\lbrace f(1),\dots,f(n)\rbrace = n\log n +Bn+o(n),
\end{equation}
where $B$ is a constant depending on $f$.\\
In this paper, we use arguments of commutative algebra and complex analysis to find a nontrivial rational divisor of $L_{c,m,n}$ ($c , m , n \in \N^*$). As a consequence, we derive some new nontrivial lower bounds for $L_{c,m,n}$. The rest of the paper is organized in four parts (subsections). In the first part, we give an algebraic lemma which allows us, on the one hand to re-demonstrate \linebreak Theorem \ref{03} of Oon by an easy and purely algebraic method, and on the other hand to reformulate the problem of bounding from below the number $L_{c,m,n}$. In that reformulation, we are leaded to introduce a vital arithmetic function, noted $h_c$, whose multiple provides a divisor for $L_{c,m,n}$. In the next two parts, we study the arithmetic function $h_c$ and we find for it a simple multiple. In the last part, we use the obtained multiple of $h_c$ to deduce a nontrivial divisor for $L_{c,m,n}$. Our new nontrivial lower bounds for $L_{c,m,n}$ then follow from that divisor.
\section{The results and the proofs}
\subsection{An algebraic method}
Although the method used by Oon \cite{oon} to obtain his result (i.e., Theorem \ref{03}) is analytic, the ingredients for its success are algebraic in depth, as we will show it below by applying the following fundamental algebraic lemma: 
\begin{lemma}\label{l1}
Let $\mathcal{A}$ be an integral domain and $n$ be a positive integer. Let also $u_0,u_1,\dots,u_n$, $a,b$ be elements of $\mathcal{A}$. Suppose that $a$ and $b$ satisfy the following conditions:
\begin{enumerate}
\item Each of the elements $u_0,u_1,\dots,u_n$ of $\mathcal{A}$ divides $a$. 
\item Each of the elements $\prod_{\begin{subarray}{c}0\leq j\leq n \\ j\neq i\end{subarray}}\left(u_i-u_j\right)$ $(i=0,1,\dots,n)$ of $\mathcal{A}$ divides $b$. 
\end{enumerate}
Then the product $ab$ is a multiple of the product $u_0u_1\cdots u_n$.
\end{lemma}    
\begin{proof}
If the elements $u_0,u_1,\dots,u_n$ of $\mathcal{A}$ are not pairwise distinct, the result of the lemma is trivial, since by its second condition we have $b=0_{\mathcal{A}}$. Suppose for the sequel that the $u_i$'s $(i=0,1,\dots,n)$ are pairwise distinct. We use the well-known result that if a polynomial in one indeterminate, with coefficients in an integral domain, has a number of roots (in that domain) greater than its degree then it is zero. Since $a$ is a multiple of each of the elements $u_0,u_1,\dots,u_n$ of $\mathcal{A}$, then there exist $k_0,k_1,\dots,k_n\in \mathcal{A}$ such that:
\begin{equation}\label{eq1}
a=k_0u_0=k_1u_1=\dots=k_nu_n.
\end{equation}   
Similarly, since $b$ is a multiple of each of the elements $\prod_{\begin{subarray}{c}0\leq j\leq n \\ j\neq i\end{subarray}}\left(u_i-u_j\right) (i=0,1,\dots,n)$, then there exist $\ell_0,\ell_1,\dots,\ell_n\in \mathcal{A}$ such that:
\begin{equation}\label{eq2}
b=\ell_i\prod_{\begin{subarray}{c}0\leq j\leq n \\ j\neq i\end{subarray}}\left(u_i-u_j\right)~~\left(\forall i\in\left\lbrace0,1,\dots,n\right\rbrace\right).
\end{equation} 
Now, consider the following polynomial of $\mathcal{A}[X]$:
\[P\left(X\right):=\sum_{i=0}^{n}\left[\ell_i\prod_{\begin{subarray}{c}0\leq j\leq n \\ j\neq i\end{subarray}}\left(X-u_j\right)\right]-b.\]
Obviously, we have $\deg P\leq n$. On the other hand, we have (according to \eqref{eq2}):
\[P\left(u_i\right)=0~~\left(\forall i\in\left\lbrace0,1,\dots,n\right\rbrace\right),\]
showing that the number of roots of $P$ in $\mathcal{A}$ is greater than its degree. So, according to the elementary result of commutative algebra announced above, the polynomial $P$ is zero. In particular, we have $P\left(0\right)=0$; that is:
\[b=(-1)^n\sum_{i=0}^{n}\ell_i\left(\prod_{\begin{subarray}{c}0\leq j\leq n \\ j\neq i\end{subarray}}u_j\right).\]
By multiplying the two sides of this last equality by $a$, we get (according to \eqref{eq1}):
\begin{align*}
ab&=(-1)^n\sum_{i=0}^{n}\ell_ia\left(\prod_{\begin{subarray}{c}0\leq j\leq n \\ j\neq i\end{subarray}}u_j\right)\\&=(-1)^n\sum_{i=0}^{n}\ell_ik_iu_i\left(\prod_{\begin{subarray}{c}0\leq j\leq n \\ j\neq i\end{subarray}}u_j\right)\\&=(-1)^n\left(\sum_{i=0}^{n}\ell_ik_i\right)u_0u_1\cdots u_n,
\end{align*}   
showing that $ab$ is a multiple of $u_0u_1\cdots u_n$, as required. This completes the proof. 
\end{proof}
Now, we use Lemma \ref{l1} to establish a new proof of Theorem \ref{03}, which is purely algebraic. 
\begin{proof}[A new proof of Theorem \ref{03}]
Since $L_{c,m,n}$ is obviously non-increasing relative to $m$, then it suffices to prove the result of the theorem for $m=\left\lceil\frac{n}{2}\right\rceil$, that is $L_{c,\left\lceil\frac{n}{2}\right\rceil,n}\geq 2^n$. For simplicity, put $m_0=\left\lceil\frac{n}{2}\right\rceil$. So, we have to show that $L_{c,m_0,n}\geq 2^n$. For $n\in\left\lbrace 1,2,\dots,6\right\rbrace$, this can be easily checked by hand (as is done by Oon). Suppose for the sequel that $n\geq 7$. It is well-known and easily proved that for any integer $r\geq 7$, we have $\left\lceil \frac{r}{2}\right\rceil\binom{r}{\left\lceil \frac{r}{2}\right\rceil}\geq 2^{r}$. According to this inequality for $r=n$, it suffices to show that $L_{c,m_0,n}\geq m_0\binom{n}{m_0}$. More generally, we shall show that:
\begin{equation}\label{cor2}
L_{c,m',n}\geq m'\binom{n}{m'}~~~~(\forall m'\in\mathbb{N^*},~m'\leq n).
\end{equation}
Let $m'\in\mathbb{N^*}$ such that $m'\leq n$. To prove \eqref{cor2}, we apply Lemma \ref{l1} for $\mathcal{A}=\mathbb{Z}[\sqrt{-c}]$ by taking for the $u_i$'s the elements $m'+\sqrt{-c},m'+1+\sqrt{-c},\dots,n+\sqrt{-c}$ of $\mathcal{A}$ and for $a$ and $b$ the integers $a=L_{c,m',n}$ and $b=(n-m')!$. For any $k\in\left\lbrace m',m'+1,\dots,n\right\rbrace$, Since $L_{c,m',n}$ is obviously a multiple (in $\mathbb{Z}$, so also in $\mathcal{A}=\mathbb{Z}[\sqrt{-c}]$) of $k^2+c$ and $k^2+c=\left(k+\sqrt{-c}\right)\left(k-\sqrt{-c}\right)$ is a multiple (in $\mathbb{Z}[\sqrt{-c}]$) of $k+\sqrt{-c}$, then $L_{c,m',n}$ is a multiple (in $\mathbb{Z}[\sqrt{-c}]$) of $k+\sqrt{-c}$. This shows that the first condition of Lemma \ref{l1} is satisfied. On the other hand, we have for all $k\in\left\lbrace m',m'+1,\dots,n\right\rbrace$: 
\[\prod_{\begin{subarray}{c}m'\leq \ell\leq n \\ \ell\neq k\end{subarray}}\left\lbrace\left(k+\sqrt{-c}\right)-\left(\ell+\sqrt{-c}\right)\right\rbrace=\prod_{\begin{subarray}{c}m'\leq \ell\leq n\\ \ell\neq k\end{subarray}}(k-\ell)=(-1)^{n-k}(k-m')!(n-k)!,\]
which divides (in $\mathbb{Z}$, so also in $\mathbb{Z}[\sqrt{-c}]$) the integer $(n-m')!$ (since $\frac{(n-m')!}{(k-m')!(n-k)!}=\binom{n-m'}{k-m'}\in\mathbb{Z}$). This shows that the second condition of Lemma \ref{l1} is also satisfied. We thus deduce (by applying Lemma \ref{l1}) that $L_{c,m',n}(n-m')!$ is a multiple (in $\mathbb{Z}[\sqrt{-c}]$) of $\prod_{k=m'}^{n}\left(k+\sqrt{-c}\right)$. So, there exist $x,y\in\mathbb{Z}$ such that:
\begin{equation}\label{**}
L_{c,m',n}(n-m')!=\left(x+y\sqrt{-c}\right)\prod_{k=m'}^{n}\left(k+\sqrt{-c}\right).\end{equation}
Then, by taking the modulus in $\mathbb{C}$ on both sides, we get
\[L_{c,m',n}(n-m')!=\sqrt{x^2+cy^{2}}\prod_{k=m'}^{n}\sqrt{k^2+c}.\]
Next, since $x^2+cy^2\in\mathbb{N}$ and $x^2+cy^2\neq 0$ (because $x^2+cy^2=0$ $\Longrightarrow$ $L_{c,m',n}=0$, which is false) then $x^2+cy^2\geq 1$. Hence
\[L_{c,m',n}=\frac{\sqrt{x^2+cy^2}\prod_{k=m'}^{n}\sqrt{k^2+c}}{(n-m')!}\geq \frac{\prod_{k=m'}^{n}\sqrt{k^2+c}}{(n-m')!}\geq \frac{\prod_{k=m'}^{n}k}{(n-m')!}=m'\binom{n}{m'},\]
as required. This completes the proof of the theorem. 
\end{proof}
 
\noindent Naturally, we have the following question:
\begin{quote}
\textit{How could we improve the Oon lower bound $L_{c,m,n}\geq \frac{\prod_{k=m}^{n}\sqrt{k^2+c}}{(n-m)!}$?.}
\end{quote}
To simplify, suppose that $c=1$ and let $m,n\in\mathbb{N^*}$ such that $m\leq n$. According to Formula \eqref{**}, the positive integer $L_{1,m,n}(n-m)!$ is a multiple (in $\mathbb{Z}[i]$) of the Gauss integer $\prod_{k=m}^{n}(k+i)$. Next, by taking the conjugates (in $\mathbb{C}$) of both sides of \eqref{**}, we obtain that $L_{1,m,n}(n-m)!$ is also a multiple (in $\mathbb{Z}[i]$) of the Gauss integer $\prod_{k=m}^{n}(k-i)$. It follows from those two facts that $L_{1,m,n}(n-m)!$ is a multiple (in $\mathbb{Z}[i]$) of:
\begin{align*}
{\lcm}_{\mathbb{Z}[i]}\left\lbrace \prod_{k=m}^{n}(k+i),\prod_{k=m}^{n}(k-i)\right\rbrace&=\frac{\prod_{k=m}^{n}(k+i)\cdot\prod_{k=m}^{n}(k-i)}{{\gcd}_{\mathbb{Z}[i]}\left\lbrace \prod_{k=m}^{n}(k+i),\prod_{k=m}^{n}(k-i)\right\rbrace}\\&=\frac{\prod_{k=m}^{n}(k^2+1)}{{\gcd}_{\mathbb{Z}[i]}\left\lbrace \prod_{k=m}^{n}(k+i),\prod_{k=m}^{n}(k-i)\right\rbrace}.
\end{align*}
Consequently
\begin{equation}\label{eq4}
L_{1,m,n}\geq \frac{\prod_{k=m}^{n}(k^2+1)}{(n-m)!\left|{\gcd}_{\mathbb{Z}[i]}\left\lbrace \prod_{k=m}^{n}(k+i),\prod_{k=m}^{n}(k-i)\right\rbrace\right|}.
\end{equation}
Remarkably, the trivial upper bound 
\[\left|{\gcd}_{\mathbb{Z}[i]}\left\lbrace \prod_{k=m}^{n}(k+i),\prod_{k=m}^{n}(k-i)\right\rbrace\right|\leq \left|\prod_{k=m}^{n}(k+i)\right|\leq \prod_{k=m}^{n}\sqrt{k^2+1}\]
suffices to establish the Oon lower bound $L_{1,m,n}\geq \frac{\prod_{k=m}^{n}\sqrt{k^2+1}}{(n-m)!}$. So, a nontrivial upper bound for the number  $\left|{\gcd}_{\mathbb{Z}[i]}\left\lbrace \prod_{k=m}^{n}(k+i),\prod_{k=m}^{n}(k-i)\right\rbrace\right|$ certainly gives an improvement of the Oon theorem. On the other hand, for $a,b\in\mathbb{Z}$ such that $(a,b)\neq (0,0)$, we can easily check that ${\gcd}_{\mathbb{Z}[i]}\left(a+bi,a-bi\right)$ is not far from ${\gcd}_{\mathbb{Z}}(a,b)$. Precisely, we have:
\[{\gcd}_{\mathbb{Z}[i]}\left(a+bi,a-bi\right)=\left(\alpha+i\beta\right){\gcd}_{\mathbb{Z}}(a,b),\]
where $\alpha,\beta\in\left\lbrace -1,0,1\right\rbrace$ and $(\alpha,\beta)\neq (0,0)$. So, for the case $c=1$, we are leaded to study the arithmetic function:
\[\begin{array}{rcl}h :~\mathbb{Z}[i]\setminus\left\lbrace 0\right\rbrace &\longrightarrow &\mathbb{N^{*}}\\a+bi ~~& \longmapsto &\gcd (a,b)\end{array},\]
and precisely to find nontrivial upper bounds for the quantities $h\left(\prod_{k=m}^{n}(k+i)\right)$ $(m,n\in\mathbb{N^*},~m\leq n)$. For the general case $(c\in\mathbb{N^*})$, the arithmetic function we need to study is clearly given by:
\[\begin{array}{rcl}h_c :~\mathbb{Z}[\sqrt{-c}]\setminus\left\lbrace 0\right\rbrace &\longrightarrow &\mathbb{N^{*}}\\a+b\sqrt{-c}~~& \longmapsto &\gcd (a,b)\end{array}\]
and the quantities we need to bound from above are $h_c\left(\prod_{k=m}^{n}(k+\sqrt{-c})\right)$ $(m,n\in\mathbb{N^*},~m\leq n)$.

The following proposition has as objective to replace a specific arithmetic language of the ring $\mathbb{Z}[\sqrt{-c}]$ by its analog (more simple) in $\mathbb{Z}$.   
\begin{prop}\label{p1}
Let $c\in\mathbb{N^*}$ and $N,a,b\in\mathbb{Z}$, with $(a,b)\neq (0,0)$. Then, $N$ is a multiple $($in $\mathbb{Z}[\sqrt{-c}])$ of $\left(a+b\sqrt{-c}\right)$ if and only if $N$ is a multiple $($in $\mathbb{Z})$ of $\frac{a^2+cb^2}{\gcd(a,b)}$.  
\end{prop}
\begin{proof}
The result of the proposition is trivial for $b=0$. Suppose for the sequel that $b\neq 0$.

Suppose that $N$ is a multiple (in $\mathbb{Z}[\sqrt{-c}]$) of $\left(a+b\sqrt{-c}\right)$; that is there exist $x,y\in\mathbb{Z}$ such that:
\[N=\left(x+y\sqrt{-c}\right)\left(a+b\sqrt{-c}\right).\]   
By identifying the real and imaginary parts of the two hand-sides of this equality, we get
\begin{align}
N&=ax-byc, \label{eq5} \\ 0&=bx+ay. \label{eq6}
\end{align}
Next, putting $d:=\gcd(a,b)$, there exist $a',b'\in\mathbb{Z}$, with $b'\neq 0$ and $\gcd(a',b')=1$, such that $a=da'$ and $b=db'$. By substituting these in \eqref{eq6}, we obtain (after simplifying):
\begin{equation}\label{eq7}
b'x=-a'y.
\end{equation}
This last equality shows that $b'$ divides $a'y$. But since $\gcd(a',b')=1$, then (according to the Gauss lemma) $b'$ divides $y$. So there exists $k\in\mathbb{Z}$ such that $y=kb'$. By reporting this in \eqref{eq7}, we get $x=-ka'$. Then, by substituting $x=-ka'=-k\frac{a}{d}$ and $y=kb'=k\frac{b}{d}$ in \eqref{eq5}, we finally obtain
\[N=-k\frac{a^2+cb^2}{d}=-k\frac{a^2+cb^2}{\gcd(a,b)},\]
showing that $N$ is a multiple (in $\mathbb{Z}$) of $\frac{a^2+cb^2}{\gcd(a,b)}$, as required.

Conversely, suppose that $N$ is a multiple (in $\mathbb{Z}$) of $\frac{a^2+cb^2}{\gcd(a,b)}$. Then, there exists $k\in\mathbb{Z}$ such that:
\[N=k\frac{a^2+cb^2}{\gcd(a,b)}=k\frac{a-b\sqrt{-c}}{\gcd(a,b)}\left(a+b\sqrt{-c}\right)=\left(k\frac{a}{\gcd(a,b)}-k\frac{b}{\gcd(a,b)}\sqrt{-c}\right)\left(a+b\sqrt{-c}\right).\]
Since $\left(k\frac{a}{\gcd(a,b)}-k\frac{b}{\gcd(a,b)}\sqrt{-c}\right)\in\mathbb{Z}[\sqrt{-c}]$, the last equality shows that $N$ is a multiple (in $\mathbb{Z}[\sqrt{-c}]$) of $\left(a+b\sqrt{-c}\right)$, as required. This completes the proof of the proposition.    
\end{proof}
From Proposition \ref{p1}, we derive the following corollary, which is the first key step to obtaining the results of this paper. 
\begin{coll}\label{p2}
Let $c,m,n\in\mathbb{N^*}$ such that $m\leq n$. Then, the positive integer $L_{c,m,n}(n-m)!$ is a multiple $($in $\mathbb{Z})$ of the positive integer:
\[\frac{\prod_{k=m}^{n}\left(k^2+c\right)}{h_c\left(\prod_{k=m}^{n}\left(k+\sqrt{-c}\right)\right)}.\] 
\end{coll}
\begin{proof}
Formula \eqref{**} (obtained during our new proof of Theorem \ref{03}) shows that $L_{c,m,n}(n-m)!$ is a multiple (in $\mathbb{Z}[\sqrt{-c}]$) of $\prod_{k=m}^{n}\left(k+\sqrt{-c}\right)$. But, according to Proposition \ref{p1}, this last property is equivalent to the statement of the corollary. 
\end{proof}
In view of Corollary \ref{p2}, to bound from below $L_{c,m,n}$ $(c,m,n\in\mathbb{N^*},~m\leq n)$, it suffices to bound from above $h_c\left(\prod_{k=m}^{n}(k+\sqrt{-c})\right)$. Likewise, to find a nontrivial (rational) divisor of $L_{c,m,n}$, it suffices to find a nontrivial multiple of $h_c\left(\prod_{k=m}^{n}(k+\sqrt{-c})\right)$. This is what we will do in what follows.
 
\subsection{An explicit Bézout identity}\label{sub22}
In the following, let $c\in\mathbb{N^*}$ and $k\in\mathbb{N}$ be fixed and define
\begin{align*}
P_{k}\left(X\right)&:=\left(X+\sqrt{-c}\right)\left(X-1+\sqrt{-c}\right)\cdots \left(X-k+\sqrt{-c}\right):=A_{k}\left(X\right)+B_{k}\left(X\right)\sqrt{-c}, \\ \overline{P_{k}}\left(X\right)&:=\left(X-\sqrt{-c}\right)\left(X-1-\sqrt{-c}\right)\cdots \left(X-k-\sqrt{-c}\right):=A_{k}\left(X\right)-B_{k}\left(X\right)\sqrt{-c},
\end{align*}
where it is understood that $A_{k},B_{k}\in\mathbb{Z}[X]$. In what follows, we find nontrivial multiples for the positive integers $h_{c}\left(P_{k}(n)\right)=\gcd\left(A_{k}(n),B_{k}(n)\right)$ $(n\geq 1)$. To do so, we look for two polynomial sequences $\left(a_{k}(n)\right)_{n}$ and $\left(b_{k}(n)\right)_{n}$ so that the polynomial sequence $\left(a_{k}(n)A_{k}(n)+b_{k}(n)B_{k}(n)\right)_{n}$ be independent on $n$. Clearly, this leads to looking for two polynomials $U_{k},V_{k}\in\mathbb{Q}[X]$ which satisfy the Bézout identity:
\[U_{k}\left(X\right)A_{k}\left(X\right)+V_{k}\left(X\right)B_{k}\left(X\right)=1.\]
Next, since $A_{k}=\frac{P_{k}+\overline{P_{k}}}{2}$ and $B_{k}=\frac{P_{k}-\overline{P_{k}}}{2\sqrt{-c}}$, the problem is equivalent to looking for $\alpha_{k},\beta_{k}\in\mathbb{Q}\left(\sqrt{-c}\right)[X]$ such that:
\[\alpha_{k}\left(X\right)P_{k}\left(X\right)+\beta_{k}\left(X\right)\overline{P_{k}}\left(X\right)=1.\]
Let us first justify the existence of such $\alpha_k$ and $\beta_k$. Denoting by $Z\left(P\right)$ the set of all the complex roots of a polynomial $P\in\mathbb{C}[X]$, we have clearly:
\[Z\left(P_{k}\right)=\left\lbrace -\sqrt{-c},1-\sqrt{-c},\dots,k-\sqrt{-c}\right\rbrace ~\text{and}~Z\left(\overline{P_{k}}\right)=\left\lbrace \sqrt{-c},1+\sqrt{-c},\dots,k+\sqrt{-c}\right\rbrace,\]
showing that $Z\left(P_{k}\right)\cap Z\left(\overline{P_{k}}\right)=\emptyset$; that is $P_{k}$ and $\overline{P_{k}}$ do not have a common root in $\mathbb{C}$. This implies that $P_{k}$ and $\overline{P_{k}}$ are coprime in $\mathbb{C}[X]$; so coprime also in $\mathbb{Q}\left(\sqrt{-c}\right)[X]$ (since $P_{k},\overline{P_{k}}\in\mathbb{Q}\left(\sqrt{-c}\right)[X]$). It follows (according to Bézout's theorem) that there exist $\alpha_{k},\beta_{k}\in\mathbb{Q}\left(\sqrt{-c}\right)[X]$ such that: $\alpha_{k}P_{k}+\beta_{k}\overline{P_{k}}=1$, as required.

Now, to find explicitly such $\alpha_k$ and $\beta_k$, we need the following more precise version of Bézout's theorem:
\begin{thm}\label{p3}
Let $\mathbb{K}$ be a field and $P$ and $Q$ be two non-constant polynomials of $\mathbb{K}[X]$ such that ${\gcd}_{\mathbb{K}[X]}\left(P,Q\right)=1$. Then, there exists a unique couple $\left(U,V\right)$ of polynomials of $\mathbb{K}[X]$, with $\deg U<\deg Q$ and $\deg V<\deg P$, such that:
\[PU+QV=1.\]
\end{thm}
\begin{proof}
Since ${\gcd}_{\mathbb{K}[X]}\left(P,Q\right)=1$, then (according to Bézout's theorem) there exist $U_0,V_0\in\mathbb{K}[X]$ such that:
\[PU_0+QV_0=1.\]
Next, consider in $\mathbb{K}[X]$ the euclidean division of $U_0$ by $Q$ and the euclidean division of $V_0$ by $\left(-P\right)$:
\begin{align*}
U_0&=U_1Q+U\\V_0&=V_1\left(-P\right)+V,
\end{align*}
where $U_1,V_1,U,V\in\mathbb{K}[X]$, $\deg U<\deg Q$ and $\deg V<\deg \left(-P\right)=\deg P$. So, we have
$$
PU+QV = P\left(U_0-U_1Q\right)+Q\left(V_0+V_1P\right) = PQ\left(V_1-U_1\right)+PU_0+QV_0 = PQ\left(V_1-U_1\right)+1 .
$$
If $V_1-U_1\neq 0$, then the last equality implies that $\deg\left(PU+QV\right)\geq \deg\left(PQ\right)$, which is impossible, since $\deg U<\deg Q$ and $\deg V<\deg P$. Thus $V_1-U_1=0$, which gives $PU+QV=1$. The existence of the couple $\left(U,V\right)$ as required by the theorem is proved. It remains to prove the uniqueness of $\left(U,V\right)$. Let $\left(U_*,V_*\right)$ another couple of polynomials of $\mathbb{K}[X]$, with $\deg U_* <\deg Q$, $\deg V_* <\deg P$ and $PU_*+QV_*=1$ and let us prove that $\left(U_*,V_*\right)=\left(U,V\right)$. We have
\begin{align*}
P\left(UV_*-U_*V\right)=\left(PU\right)V_*-\left(PU_*\right)V=\left(1-QV\right)V_*-\left(1-QV_*\right)V=V_*-V,
\end{align*}
showing that the polynomial $\left(V_*-V\right)$ is a multiple of $P$ in $\mathbb{K}[X]$. But since $\deg\left(V_*-V\right)<\deg P$ (because $\deg V<\deg P$ and $\deg V_*<\deg P$), we have inevitably $V_*-V=0$; hence $V_*=V$. Using this, we get $PU_*=1-QV_*=1-QV=PU$. Thus $U_*=U$. Consequently, we have $\left(U_*,V_*\right)=\left(U,V\right)$, as required. This completes the proof of the theorem. 
\end{proof}
In our context, the application of Theorem \ref{p3} gives the following corollary:
\begin{coll}\label{c1}
There exists a unique polynomial $\alpha_k\in\mathbb{C}[X]$, with degree $\leq k$, such that:
\[\alpha_{k}P_{k}+\overline{\alpha_{k}}\overline{P_{k}}=1.\]
\end{coll}
\begin{proof}
According to Theorem \ref{p3} (applied for $\mathbb{K}=\mathbb{C}$ and $\left(P,Q\right)=\left(P_k,\overline{P_k}\right)$), there exists a unique couple $\left(\alpha_k,\beta_k\right)$ of polynomials of $\mathbb{C}[X]$, with $\deg \alpha_k<\deg \overline{P_k}=k+1$ and $\deg \beta_k<\deg P_k=k+1$, such that $\alpha_{k}P_{k}+\beta_{k}\overline{P_{k}}=1$. By taking the conjugates in $\mathbb{C}[X]$ of both sides of the last equality, we derive that $\overline{\alpha_{k}}\overline{P_{k}}+\overline{\beta_{k}}P_{k}=1$, that is $\overline{\beta_{k}}P_{k}+\overline{\alpha_{k}}\overline{P_{k}}=1$. Since $\deg \overline{\beta_k}=\deg \beta_k<k+1$ and $\deg \overline{\alpha_k}=\deg \alpha_k<k+1$, this shows that the couple $\left(\overline{\beta_k},\overline{\alpha_k}\right)$ satisfies the characteristic property of the couple $\left(\alpha_k,\beta_k\right)$. Thus $\left(\overline{\beta_k},\overline{\alpha_k}\right)=\left(\alpha_k,\beta_k\right)$, that is $\beta_k=\overline{\alpha_k}$. Consequently, we have $\alpha_{k}P_{k}+\overline{\alpha_{k}}\overline{P_{k}}=1$. This completes the proof of the corollary.   
\end{proof}
Now, we are going to determine the explicit expression of the polynomial $\alpha_k$ announced by Corollary \ref{c1}. By replacing, in the identity $\alpha_{k}\left(X\right)P_{k}\left(X\right)+\overline{\alpha_{k}}\left(X\right)\overline{P_{k}}\left(X\right)=1$, the indeterminate $X$ by the numbers $s+\sqrt{-c}$ $(s=0,1,\dots,k)$, we get  
\begin{equation}\label{dec1}
\alpha_k\left(s+\sqrt{-c}\right)=\frac{1}{P_k\left(s+\sqrt{-c}\right)}~~(\forall s\in\left\lbrace 0,1,\dots,k\right\rbrace).
\end{equation}
(since $\overline{P_{k}}\left(s+\sqrt{-c}\right)=0$ for $s=0,1,\dots,k$). So the values of $\alpha_{k}$ are known for $(k+1)$ equidistant points with distance $1$. Since $\deg\alpha_k\leq k$, this is sufficient to determine the expression of $\alpha_{k}\left(X\right)$ by using for example the Newton forward interpolation formula. Doing so, we obtain that: 
\[\alpha_k\left(X\right)=\sum_{\ell=0}^{k}\frac{\left(\Delta^{\ell}\alpha_k\right)\left(\sqrt{-c}\right)}{\ell!}\left(X-\sqrt{-c}\right)^{\underline{\ell}}.\]
Then, by using \eqref{dece}, we derive that:
\begin{align*}
\alpha_{k}\left(X\right)&=\sum_{\ell=0}^{k}\sum_{j=0}^{\ell}\frac{(-1)^{\ell-j}}{\ell!}\binom{\ell}{j}\alpha_{k}\left(j+\sqrt{-c}\right)\left(X-\sqrt{-c}\right)^{\underline{\ell}}\\
&=\sum_{\ell=0}^{k}\left\lbrace\frac{1}{\ell!}\sum_{j=0}^{\ell}(-1)^{\ell-j}\binom{\ell}{j}\alpha_{k}\left(j+\sqrt{-c}\right)\right\rbrace\left(X-\sqrt{-c}\right)^{\underline{\ell}}\\
&=\sum_{\ell=0}^{k}\left\lbrace\frac{1}{\ell!}\sum_{j=0}^{\ell}(-1)^{\ell-j}\binom{\ell}{j}\frac{1}{P_{k}\left(j+\sqrt{-c}\right)}\right\rbrace\left(X-\sqrt{-c}\right)^{\underline{\ell}}
\end{align*}  
(according to \eqref{dec1}). So, by setting for all $\ell \in\left\lbrace 0,1,\dots,k\right\rbrace$:
\begin{equation}\label{dec2}
\Theta_{k,\ell}:=\frac{1}{\ell!}\sum_{j=0}^{\ell}(-1)^{\ell-j}\binom{\ell}{j}\frac{1}{P_{k}(j+\sqrt{-c})},
\end{equation}
we get
\begin{equation}\label{dec3}
\alpha_{k}\left(X\right)=\sum_{\ell=0}^{k}\Theta_{k,\ell}\left(X-\sqrt{-c}\right)^{\underline{\ell}}.
\end{equation}
It remains to simplify the expressions of the numbers $\Theta_{k,\ell}$ $(0\leq\ell\leq k)$. To do so, we introduce the rational functions $R_{k,\ell}$ $(0\leq\ell\leq k)$, defined by:
\begin{equation}\label{dec4}
R_{k,\ell}(z):=\frac{1}{\ell!}\sum_{j=0}^{\ell}(-1)^{\ell-j}\binom{\ell}{j}\frac{1}{P_{k}(z+j+\sqrt{-c})},
\end{equation}
so that we have
\begin{equation}\label{dec5}
\Theta_{k,\ell}=R_{k,\ell}(0)~~~~(\forall \ell\in\left\lbrace 0,1,\dots,k\right\rbrace).
\end{equation}
The common domain of holomorphy of the functions $R_{k,\ell}$ $(0\leq\ell\leq k)$ is clearly the open connected region $D$ of $\mathbb{C}$, given by:
\[D:=\mathbb{C}\setminus\lbrace j - 2 \sqrt{-c} ~;~ j \in \Z \text{ and} -k\leq j\leq k\rbrace.\]
Using the principle of analytical continuation together with the theory of the gamma and beta functions, we can find another expression of $R_{k,\ell}$ $(0\leq\ell\leq k)$, which is simpler than the above. We have the following proposition:
\begin{prop}\label{pp}
For all $\ell \in \mathbb{N}$, with $\ell\leq k$, and all $z\in D$, we have:
\begin{equation}\label{dec6}
R_{k,\ell}(z)=\frac{(-1)^{k+\ell}}{z+2\sqrt{-c}}\binom{k+\ell}{\ell}\frac{1}{\left(k-2\sqrt{-c}-z\right)^{\underline{k}}\left(\ell+2\sqrt{-c}+z\right)^{\underline{\ell}}}.
\end{equation}
\end{prop}
\begin{proof}
Let $\ell \in \mathbb{N}$ such that $\ell \leq k$. According to the principle of analytical continuation, it suffices to prove Formula \eqref{dec6} for $z\in\mathbb{C}$, such that $\Re(z)>k$. For a such $z$, we have: \pagebreak
\begin{align*}
R_{k,\ell}(z)&:=\frac{1}{\ell!}\sum_{j=0}^{\ell}(-1)^{\ell-j}\binom{\ell}{j}\frac{1}{P_{k}(z+j+\sqrt{-c})}\\
&=\frac{1}{\ell!}\sum_{j=0}^{\ell}(-1)^{\ell-j}\binom{\ell}{j}\frac{1}{\left(z+j+2\sqrt{-c}\right)\left(z+j-1+2\sqrt{-c}\right)\cdots\left(z+j-k+2\sqrt{-c}\right)}\\
&=\frac{1}{\ell!}\sum_{j=0}^{\ell}(-1)^{\ell-j}\binom{\ell}{j}\frac{\Gamma\left(z+j-k+2\sqrt{-c}\right)}{\Gamma\left(z+j+1+2\sqrt{-c}\right)}\\
&=\frac{1}{\ell!}\sum_{j=0}^{\ell}(-1)^{\ell-j}\binom{\ell}{j}\frac{1}{k!}\beta\left(z+j-k+2\sqrt{-c},k+1\right)\\
&=\frac{1}{k!\ell!}\sum_{j=0}^{\ell}\left[(-1)^{\ell-j}\binom{\ell}{j}\int_{0}^{1}t^{z+j-k-1+2\sqrt{-c}}(1-t)^{k}\mathrm{d}t\right] \\
&=\frac{1}{k!\ell!}\int_{0}^{1} t^{z-k-1+2\sqrt{-c}}(1-t)^{k}\left\lbrace \sum_{j=0}^{\ell}(-1)^{\ell-j}\binom{\ell}{j}t^{j}\right\rbrace \mathrm{d}t \\
&=\frac{1}{k!\ell!}\int_{0}^{1} t^{z-k-1+2\sqrt{-c}}(1-t)^{k}\left(t-1\right)^{\ell} \mathrm{d}t \\
&=\frac{(-1)^{\ell}}{k!\ell!}\int_{0}^{1} t^{z-k-1+2\sqrt{-c}}(1-t)^{k+\ell}\mathrm{d}t \\
&=\frac{(-1)^\ell}{k!\ell!}\beta\left(z-k+2\sqrt{-c},k+\ell+1\right)\\
&=\frac{(-1)^\ell}{k!\ell!}\frac{\Gamma\left(z-k+2\sqrt{-c}\right)\Gamma\left(k+\ell+1\right)}{\Gamma\left(z+\ell+1+2\sqrt{-c}\right)}\\
&=(-1)^\ell\binom{k+\ell}{\ell}\frac{1}{\left(z+\ell+2\sqrt{-c}\right)\left(z+\ell-1+2\sqrt{-c}\right)\cdots\left(z-k+2\sqrt{-c}\right)}\\
&=\frac{(-1)^{k+\ell}}{z+2\sqrt{-c}}\binom{k+\ell}{\ell}\frac{1}{\left(k-2\sqrt{-c}-z\right)^{\underline{k}}\left(\ell+2\sqrt{-c}+z\right)^{\underline{\ell}}},
\end{align*}
as required. This completes the proof.
\end{proof}

From Proposition \ref{pp}, we immediately derive a simpler explicit expression of $\alpha_k\left(X\right)$. We have the following corollary:
\begin{coll}\label{jan1}
We have:
\[\alpha_{k}\left(X\right)=\frac{1}{2\sqrt{-c}\left(k-2\sqrt{-c}\right)^{\underline{k}}}\sum_{\ell=0}^{k}\frac{(-1)^{k+\ell}\binom{k+\ell}{\ell}}{\left(\ell+2\sqrt{-c}\right)^{\underline{\ell}}}\left(X-\sqrt{-c}\right)^{\underline{\ell}}.\]
\end{coll}
\begin{proof}
This immediately follows from Formulas \eqref{dec3}, \eqref{dec5} and \eqref{dec6}. 
\end{proof}

\subsection{Nontrivial multiples of some values of $h_c$}

In this subsection, we preserve the notations of Subsection \ref{sub22}. From Corollary \ref{jan1}, we derive the following theorem:
\begin{thm}\label{jan2}
For all $c,n,m\in\mathbb{N^*}$, with $m\leq n$, we have:
\[h_c\left(\prod_{\ell=m}^{n}\left(\ell+\sqrt{-c}\right)\right)~~\text{divides}~~c\prod_{\ell=1}^{n-m}(\ell^2+4c).\] 
\end{thm} 
\begin{proof}
Let $c,n,m\in\mathbb{N^*}$, with $m\leq n$. Putting $k:=n-m\in\mathbb{N}$ and $d:=c\prod_{\ell=1}^{n-m}(\ell^2+4c)\in\mathbb{N^*}$, we have $\prod_{\ell=m}^{n}\left(\ell+\sqrt{-c}\right)=P_k(n)$; so, we have to show that $h_c\left(P_k(n)\right)$ divides $d$. By noting that $2d=\sqrt{-c}\cdot 2\sqrt{-c}\left(k-2\sqrt{-c}\right)^{\underline{k}}\left(k+2\sqrt{-c}\right)^{\underline{k}}$, we derive from Corollary \ref{jan1} that $2d\alpha_k\in\mathbb{Z}[\sqrt{-c}][X]$. So, there exist $r_k,s_k\in\mathbb{Z}[X]$ such that:
\[2d\alpha_k\left(X\right)=r_k\left(X\right)+s_k\left(X\right)\sqrt{-c}.\] 
Next, the identity of polynomials $\alpha_kP_k+\overline{\alpha_k}\overline{P_k}=1$ (given by Corollary \ref{c1}) implies that $2d\alpha_k\cdot P_k+\overline{2d\alpha_k}\cdot \overline{P_k}=2d$. By substituting in this last equality $P_k$ by $\left(A_k+B_k\sqrt{-c}\right)$ and $2d\alpha_k$ by $\left(r_k+s_k\sqrt{-c}\right)$, we obtain (in particular) that:
\[r_kA_k-cs_kB_k=d,\] 
implying that ${\gcd}_{\mathbb{Z}[X]}\left(A_k,B_k\right)$ divides $d$. We then conclude that ${\gcd}_{\mathbb{Z}}\left(A_k(n),B_k(n)\right)=h_c\left(P_k(n)\right)$ divides $d$, as required.
\end{proof}

\subsection{New estimates for the number $L_{c,m,n}$}

We have the following theorem:
\begin{thm}\label{t7}
Let $c,m,n\in\mathbb{N^*}$ such that $m\leq n$. Then:
\begin{enumerate}
\item The positive integer $L_{c,m,n}$ is a multiple of the rational number
\[\frac{\displaystyle\prod_{k=m}^{n}\left(k^2+c\right)}{c\cdot (n-m)!\displaystyle\prod_{k=1}^{n-m}\left(k^2+4c\right)}.\]
\item We have
\[L_{c,m,n}\geq \lambda_1(c) \cdot m^2\frac{n!^2}{m!^2(n-m)!^3},\]
where $\lambda_1(c):=e^{-\frac{2\pi^{2}}{3} c}/c$.
\end{enumerate}
\end{thm}
\begin{proof}
The first point of the theorem is an immediate consequence of Corollary \ref{p2} and Theorem \ref{jan2}. The second one follows from the first one and the easy inequalities:
\begin{align*}
\frac{\prod_{k=m}^{n}\left(k^2+c\right)}{c\cdot (n-m)!\prod_{k=1}^{n-m}\left(k^2+4c\right)}&\geq \frac{\prod_{k=m}^{n}k^2}{c\cdot (n-m)!^3\prod_{k=1}^{n-m}\left(1+\frac{4c}{k^2}\right)}\\&\geq \frac{m^2\left(\frac{n!}{m!}\right)^2}{c\cdot (n-m)!^3e^{\sum_{k=1}^{+\infty}\frac{4c}{k^2}}}\\&=\frac{e^{-\frac{2\pi^{2}}{3} c}}{c}\cdot m^2\frac{n!^2}{m!^2(n-m)!^3}.
\end{align*}
This completes the proof.
\end{proof}

We shall now impose conditions on $m$ (in terms of $n$) in order to optimize (resp. simplify) the estimate of the second point of Theorem \ref{t7}. To do so, we first need to get rid of the factorials in that estimate. We have the following:
\begin{coll}\label{t9}
Let $c,n,m\in\mathbb{N^*}$ such that $m<n$. Then, we have:
\begin{equation}\label{jan3}
L_{c,m,n}\geq \lambda_{2}(c)\cdot\frac{nm}{\left(n-m\right)^{3/2}}\left( \frac{m^2}{(n-m)^3}\right)^{n-m} e^{3(n-m)},
\end{equation}
where $\lambda_2(c):=\frac{e^{-\frac{2\pi^{2}}{3} c -\frac{5}{12}}}{\left(2\pi\right)^{3/2}c}$.
\end{coll}
\begin{proof}
Starting from the lower bound established by Theorem \ref{t7} for $L_{c , m , n}$ and estimating each of its factorial terms by using the well-known double inequality:
\[k^{k}e^{-k}\sqrt{2\pi k}\leq k!\leq k^{k}e^{-k}\sqrt{2\pi k} e^{\frac{1}{12k}}~~~~(\forall k\in\mathbb{N^*})\]
(which can be found in Problem 1.15 of \cite{kon}), we get
$$
L_{c , m , n} \geq \lambda_1(c) (2 \pi)^{- 3/2} \cdot \frac{n m}{(n - m)^{3 / 2}} \cdot \left(\frac{n}{m}\right)^{2 n} \cdot \left(\frac{m^2}{(n - m)^3}\right)^{n - m} e^{n - m} \cdot e^{- \frac{1}{6 m} - \frac{1}{4 (n - m)}} .
$$
Next, since $e^{- \frac{1}{6 m} - \frac{1}{4 (n - m)}} \geq e^{- \frac{1}{6} - \frac{1}{4}} = e^{- \frac{5}{12}}$ and $\left(\frac{n}{m}\right)^{2 n} = e^{- 2 n \log(\frac{m}{n})} \geq e^{- 2 n (\frac{m}{n} - 1)} = e^{2 (n - m)}$, then we deduce that:
$$
L_{c , m , n} \geq \lambda_1(c) (2 \pi)^{- 3/2} e^{- 5 / 12} \cdot \frac{n m}{(n - m)^{3 / 2}} \left(\frac{m^2}{(n - m)^3}\right)^{n - m} e^{3 (n - m)} ,
$$
as required.
\end{proof}

In the context of Corollary \ref{t9}, by supposing that $n-m$ is of order of magnitude $n^\alpha$ for large $n$ (where $0<\alpha<1$), then the dominant part of the lower bound \eqref{jan3} for $L_{c , m , n}$ is $\left(\frac{m^2}{(n-m)^3}\right)^{n-m}$ and has order of magnitude $n^{(2-3\alpha)n^{\alpha}}$. So, to have an optimal estimate, we must take $\alpha$ less than but not too far from $\frac{2}{3}$ (a study of the function $\alpha\longmapsto (2-3\alpha)n^{\alpha}$ shows that the best value of $\alpha$ is $\alpha=\frac{2}{3}-\frac{1}{\log n}$). A concrete result specifying this heuristic reasoning is given by the following theorem: 
\begin{thm}\label{c5}
Let $c,m,n\in\mathbb{N^*}$ such that $m\leq n - \frac{1}{2}n^{2/3}$. Then, we have:
\[L_{c,m,n}\geq \lambda_3(c)\cdot\left(n-\frac{1}{2}n^{2/3}\right)\cdot\left(2e^{3}\right)^{\left\lfloor \frac{1}{2}n^{2/3}\right\rfloor},\]
where $\lambda_3(c):=\frac{e^{-\frac{2\pi^2}{3} c -\frac{5}{12}}}{\pi^{3/2} c}$.
\end{thm}
\begin{proof}
A simple calculation shows that the result of the theorem is true for $n < 3$. Suppose for the sequel that $n \geq 3$ and let $m_n:=n-\left\lfloor \frac{1}{2}n^{2/3}\right\rfloor < n$; so $m \leq m_n$. From Corollary \ref{t9}, we have: 
\begin{align*}
L_{c,m_n,n}&\geq \lambda_2(c)\frac{n\left(n-\left\lfloor \frac{1}{2}n^{2/3}\right\rfloor\right)}{\left\lfloor \frac{1}{2}n^{2/3}\right\rfloor^{3/2}}\left( \frac{\left(n-\left\lfloor \frac{1}{2}n^{2/3}\right\rfloor\right)^2}{\left\lfloor \frac{1}{2}n^{2/3}\right\rfloor^3}\right)^{\left\lfloor \frac{1}{2}n^{2/3}\right\rfloor} e^{3\left\lfloor \frac{1}{2}n^{2/3}\right\rfloor}\\&\geq \lambda_2(c)\frac{n\left(n-\frac{1}{2}n^{2/3}\right)}{\left(\frac{1}{2}n^{2/3}\right)^{3/2}}\left(\frac{\left(n-\frac{1}{2}n^{2/3}\right)^2}{\left(\frac{1}{2}n^{2/3}\right)^3}\right)^{\left\lfloor \frac{1}{2}n^{2/3}\right\rfloor}e^{3\left\lfloor \frac{1}{2}n^{2/3}\right\rfloor}\\&=2^{3/2}\lambda_2(c)\left(n-\frac{1}{2}n^{2/3}\right)\left[8\left(1-\frac{1}{2n^{1/3}}\right)^2\right]^{\left\lfloor \frac{1}{2}n^{2/3}\right\rfloor}e^{3\left\lfloor \frac{1}{2}n^{2/3}\right\rfloor}.
\end{align*}
But since $1-\frac{1}{2n^{1/3}}\geq \frac{1}{2}$ (because $n\geq 1$), we deduce that:
\[L_{c,m_n,n}\geq 2^{3/2}\lambda_2(c)\left(n-\frac{1}{2}n^{2/3}\right)\left(2e^3\right)^{\left\lfloor \frac{1}{2}n^{2/3}\right\rfloor}.\] 
The required result follows from the trivial fact that $L_{c,m,n}\geq L_{c,m_n,n}$ (since $m\leq m_n$).
\end{proof}

In another direction, we derive from Corollary \ref{t9} the following theorem, which completes (in a way) Theorem \ref{c5} above.
\begin{thm}
Let $c,m,n\in\mathbb{N^*}$ such that $n-\frac{1}{2}n^{2/3} \leq m \leq n$. Then, we have:
\[L_{c,m,n}\geq \lambda_2(c)\cdot ne^{3(n-m)},\]
where $\lambda_2(c)$ is defined in Corollary \ref{t9}. 
\end{thm}
\begin{proof}
The result of the theorem is trivial for $m = n$. Suppose for the sequel that $m < n$; so we have $n \geq 2$. Now, let $f:\left[0,n\right]\longrightarrow \mathbb{R}$ be the function defined by $f(x)=x^2-(n-x)^3$ $(\forall x\in [0,n])$. Obviously, $f$ is increasing. Next, we have:
\begin{align*}
f\left(n-\frac{1}{2}n^{2/3}\right)&=\left(n-\frac{1}{2}n^{2/3}\right)^2-\left(\frac{1}{2}n^{2/3}\right)^3\\&=n^2-n^{5/3}+\frac{1}{4}n^{4/3}-\frac{1}{8}n^2\\&=\frac{7}{8}n^2-n^{5/3}+\frac{1}{4}n^{4/3}.
\end{align*}
But since $n^2\geq \frac{8}{7}n^{5/3}$ (because $n\geq 2$), it follows that $f\left(n-\frac{1}{2}n^{2/3}\right)\geq \frac{1}{4}n^{4/3}>0$. So, the increase of $f$ insures that $f(m)>0$ (since $m\geq n-\frac{1}{2}n^{2/3}$ by hypothesis). Thus $\frac{m^2}{(n-m)^3}>1$ and $\frac{m}{(n-m)^{3/2}}>1$. By reporting these into \eqref{jan3}, we then conclude that:
\[L_{c,m,n}\geq \lambda_2(c)\cdot ne^{3(n-m)},\]
as required. This completes the proof.  
\end{proof}

\end{document}